\documentclass[a4paper,12pt]{article}
\usepackage[utf8]{inputenc}
\usepackage{amsmath,amssymb,amsthm,mathrsfs}
\usepackage{graphicx}
\usepackage{cite}
\usepackage{hyperref}
\hypersetup{
    colorlinks=true,
    linkcolor=blue
    }

\usepackage{pdflscape}
\usepackage{adjustbox}
\usepackage{setspace}
\usepackage{cellspace}

\usepackage{enumitem}
\usepackage[utf8]{inputenc}

 \usepackage{graphicx}
\usepackage{array}
\usepackage{geometry}
\usepackage{mathrsfs}

\theoremstyle{definition} \newtheorem{remark}{Remark}[section]
\theoremstyle{plain}  \newtheorem{theorem}{Theorem}[section]
\theoremstyle{plain}  \newtheorem{lemma}{Lemma}[section]
\theoremstyle{plain}  
\theoremstyle{plain}  \newtheorem{corollary}{Corollary}[section]
\theoremstyle{definition} \newtheorem{definition}{Definition}[section]
\theoremstyle{remark} \newtheorem{example}{Example}[section]

\numberwithin{equation}{section}

\def \MRA{multiresolution analysis}

\title{\textbf{Convergence of Subdivision Schemes and Smoothness of Refinable Functions on $ p $-adic Fields}}
\author{\textbf{Athira N$ ^{\ast} $ and Lineesh M C$^{\dagger} $}}
\date{}

\begin{document}

\maketitle
\begin{center}
\small{Department of Mathematics \\
National Institute of Technology Calicut \\
NIT Campus P O - 673 601, India \\
$ ^{\ast} $nathira.1996@gmail.com \\
$ ^{\dagger} $lineesh@nitc.ac.in}
\end{center}

\begin{abstract}
    A systematic and comprehensive study of $ p $-adic refinement equations and subdivision scheme associated with a finitely supported refinement mask are carried out in this paper. The $ L_{q} $-convergence of the subdivision scheme is characterized in terms of the $ q $-norm joint spectral radii of a collection of operators associated with the refinement mask. Also, the smoothness of complex-valued functions on $ \mathbb{Q}_{p} $\, is investigated.
\end{abstract}
\textbf{keywords:} $ p $-adic field, pseudo-differential operator, $ p $-adic refinable function, $ p $-adic wavelets, subdivision schemes, joint spectral radius, compact support, Lipschitz spaces, H\"{o}lder spaces, smoothness. \\
    

 \section{Introduction}
Wavelets has been a very popular topic of conversation in many scientific and engineering gatherings nowadays. The goal of wavelet research is to create a set of basis functions that will give an informative, efficient, and useful description of a function or signal. Alf\'{r}ed Haar constructed the first wavelet in 1910. There after a lot of mathematicians including Yves Meyer, Stephane Mallat, and Ingrid Daubechies have constructed different kinds of wavelets to serve on many applications. The concept of wavelets is then extended to the Euclidean space as well as many other topological spaces such as $ p $-adic fields. \par 
Wavelets can be generated from scaling functions as well as wavelet sets. In their paper, Albeverio \textit{et al.}\cite{aes1,aes2} proposed a complete characterization of scaling functions and explained what types of scaling functions form $ p $-adic multiresolution analysis. Khrennikov \textit{et al.} \cite{kss2} described the procedure to construct $ p $-adic wavelets from $ p $-adic scaling functions associated with an expansive automorphism. The authors \cite{kss1} discussed about all compactly supported orthogonal wavelet bases for $ L_{2}(\mathbb{Q}_{p}) $\, generated by the unique p-adic multiresolution analysis, i.e., the Haar bases of $ L_{2}(\mathbb{Q}_{p}) $. \par 
In 2008, Shelkovich and Skopina \cite{ss} constructed infinitely many different multidimensional $ p $-adic Haar orthonormal wavelet bases for $ L_{2}(\mathbb{Q}_{p}^{n}) $\, and in 2009, Khrennikov and Shelkovich \cite{ks} developed infinite family of compactly supported non-Haar type $ p $-adic wavelet bases for $ L_{2}(\mathbb{Q}_{p}^{n}),\, n\geq 1 $. A method for constructing MRA-based $ p $-adic wavelet systems that form Riesz bases in $ L_{2}(\mathbb{Q}_{p}) $\, is developed in \cite{aes3}. \par 
The approximation and smoothness properties of wavelets are determined by the corresponding refinable functions. In their paper Athira and Lineesh \cite{al} discussed about the approximation order of shift-invariant space of a reﬁnable function on p-adic ﬁelds. The relation between the approximation order, accuracy of refinable function and order of the Strang–Fix condition are also established. \par 
Cavaretta, Dahmen, and Micchelli \cite{cdm} proved that the subdivision scheme associated with a refinement mask converges uniformly, provided that the normalized solution of the corresponding refinement equation is continuous and has stable shifts. In 1998, Bin Han and Jia \cite{bj} investigated about multivariate refinement equations associated with a dilation matrix and a finitely supported refinement mask. They characterize the $ L_{p} $-convergence of a subdivision scheme in terms of the $ p $-norm joint spectral radius of a collection of matrices associated with the refinement mask. Then Jia\cite{rqj2} characterized the optimal smoothness of a multivariate refinable function in terms of the spectral radius of the corresponding transition operator restricted to a suitable finite dimensional invariant subspace. In 1999, Jia, Riemenschneider and Zhou \cite{jrz} investigated the existence of multiple refinable functions using convergence of subdivision scheme and provided the smoothness properties of multiple refinable functions and multiple wavelets. \par 
Our goal is to extend the concept of subdivision scheme and smoothness to the $ p $-adic field $ \mathbb{Q}_{p} $. Section 2 contains preliminary information about scaling functions, wavelet functions and accuracy of scaling functions on $ \mathbb{Q}_{p} $. Section 3 is devoted to a study of joint spectral radii of a finite collection of linear operators associated to a refinement equation. Characterization of the $ L_{q} $-convergence of a subdivision scheme in terms of the $ q $-norm joint spectral radius of a collection of matrices associated with the refinement mask is given in section 4. We also provide some examples to illustrate the proposed theory. Section 5 contains a detailed analysis of smoothness of complex valued functions on $ \mathbb{Q}_{p} $.

\section{Preliminary}
Let $ p $\, be a prime number. Consider the completion field of $ \mathbb{Q} $\, with respect to the norm  $ \lvert \cdot  \rvert_{p} $\, defined by,
$$ \lvert x \rvert_{p} = \begin{cases}
                          0 & ;\, x = 0 \\
                          p^{-\gamma} & ;\,  x \neq 0, x= (p^{\gamma})\frac{m}{n},
                         \end{cases}
 $$
 where $ \gamma = \gamma(x) \in \mathbb{Z}, \, m,n \in \mathbb{Z} $\, not divisible by $ p $. Denote the above field as 
 $ G = \mathbb{Q}_{p} $. The canonical form of $ x \in \mathbb{Q}_{p}$, $ x \neq 0 $\, is, 
 \begin{equation}\label{eq1.1}
  x=p^{\gamma}(x_{0}+x_{1}p+x_{2}p^{2}+\cdots ),
 \end{equation}
where $ \gamma \in \mathbb{Z}, \, x_{j} \in \{ 0,1,\ldots, p-1\}, \, x_{0} \neq 0 $. Then for this $ x \in \mathbb{Q}_{p} $, 
the fractional part of $ x $\, is, 
$$ \{ x \}_{p} =\begin{cases}
		  0 & ;\, \gamma(x) \geq 0 \text { or } x = 0 \\
		  p^{\gamma}(x_{0}+x_{1}p+\cdots + x_{-\gamma-1}p^{-\gamma-1}) & ;\, \gamma(x) <0.
		\end{cases}
 $$ 
In $ \mathbb{Q}_{p} $, $ -1 $\, can be written as $$ -1 = (p-1)+(p-1)p+(p-1)p^{2}+ \cdots. $$ \par 
 The dual group of $ \mathbb{Q}_{p} $\, is $ \mathbb{Q}_{p} $\, itself and the character on $ \mathbb{Q}_{p} $\, is defined as,
 \begin{equation}\label{eq1.5}
  \chi(x, \xi) = e^{2\pi i \{x\xi\}_{p}},
 \end{equation}
  where $ \{\cdot \}_{p} $\, is the fractional part of a number. Denote 
 $$ B_{\gamma}(a) = \{ x \in \mathbb{Q}_{p} : \lvert x-a \rvert_{p} \leq p^{\gamma} \}. $$ Then $ B_{0}(0) $\, is a compact 
 open subgroup of $ \mathbb{Q}_{p} $. Let $ \mu $\, be the Haar measure on $ \mathbb{Q}_{p} $\, with $ \mu(B_{0}(0)) = 1 $. 
 Denote $ d\mu(x) $\, by $ dx $. \par 
 For $ 1 \leq q \leq \infty $, by $ L_{q}(\mathbb{Q}_{p}) $\, we denote the Banach space of all (complex-valued) measurable functions $ f $\, on $ \mathbb{Q}_{p} $\, such that $ \lVert f \rVert_{q} < \infty $, where
 $$ \lVert f \rVert_{q}:= (\int_{\mathbb{Q}_{p}} \lvert f(x) \rvert^{q} dx )^{1/q} \qquad \text{for } 1 \leq q < \infty. $$ and $ \lVert f \rVert_{\infty} $\, is the essential supremum of $ f $\, on $ \mathbb{Q}_{p} $. The Fourier transform a complex-valued function $ f $\, defined on $ \mathbb{Q}_{p} $\, is defined as 
 $$ \widehat{f}(\xi) := \int_{\mathbb{Q}_{p}} f(x) \chi(x, \xi) dx. $$ \par 
  \par 
 For the $ p $-adic analysis related to the mapping $ \mathbb{Q}_{p} \rightarrow \mathbb{C} $, the operation of differentiation is not defined. An analogy of the differentiation operator is a pseudo-differential operator. The pseudo-differential operator $ D^{\alpha} : \phi \rightarrow D^{\alpha}\phi $\, is defined 
\cite{vvz} by
\begin{equation}\label{eq1.2}
 D^{\alpha}\phi = f_{-\alpha} * \phi 
\end{equation}
where $ f_{\alpha}(x) = \frac{\lvert x \rvert_{p}^{\alpha -1}}{\Gamma_{p}(\alpha)} $\, with $ \Gamma_{p}(\alpha) = 
\frac{1-p^{\alpha-1}}{1-p^{-\alpha}} $. 
\begin{remark}
  For $ \alpha \in \mathbb{R} $\, and $ a \in \mathbb{Q}_{p} \setminus \{0\} $, $ D^{\alpha}\chi(a,x) = \lvert a 
  \rvert_{p}^{\alpha} \chi(a,x) $.
 \end{remark}
 \begin{remark}
 For $ \alpha \in \mathbb{R}, \, \gamma \in \mathbb{Z} $, let $ \Phi(x) = \delta(\lvert \xi \rvert_{p} - 
  p^{\gamma})f(\xi) $. Then $$ D^{\alpha}(\widehat{\Phi(x)}) = p^{\gamma \alpha} \widehat{\Phi(x)}. $$
\end{remark} 
Let us consider the set 
$$ I_{p} = \{ a = p^{-\gamma}(a_{0}+a_{1}p+ \cdots + a_{\gamma-1}p^{\gamma-1}) : \gamma \in \mathbb{N},\, a_{j} \in 
\{ 0,1,\ldots, p-1 \} \}. $$
Then there is a ``natural'' decomposition of $ \mathbb{Q}_{p} $\, into a union of mutually disjoint discs:
$ \mathbb{Q}_{p} = \bigcap_{a \in I_{p}} B_{0}(a) $. So, $ I_{p} $\, is a ``natural set'' of shifts for $ \mathbb{Q}_{p} $. We can see that $ I_{p}$\, is not a group. But $ I_{p} \approx \mathbb{Q}_{p}/B_{0}(0) $\, and $\mathbb{Q}_{p}/B_{0}(0) $\, is a group. From the next section, the operation of addition in $ I_{p} $\, is defined as $ \alpha + \beta (\text{mod } B_{0}(0)) $.  \par 
We denote by $ l(I_{p}) $\, the linear space of all sequences on $ I_{p} $, and by $ l_{0}(I_{p}) $\, the linear space of all 
finitely supported sequences on $ I_{p} $. For $ \alpha \in I_{p} $, we denote by $ \delta_{\alpha} $, the element in 
$ l_{0}(I_{p}) $\, given by 
$$ \delta_{\alpha}(\beta) = \begin{cases}
                             1 & \text{if } \beta =\alpha \\
                             0 & \text{if } \beta \in I_{p}\setminus\{\alpha\}.
                            \end{cases}
$$
In particular, we write $ \delta = \delta_{0} $. For $ y \in I_{p} $, we use $ \nabla_{y} $\, to denote the difference operator on $ l(I_{p}) $\, given by 
$$ \nabla_{y}u = u-u(\cdot-y), \qquad u \in l(I_{p}). $$ \par 
Define $ A:\mathbb{Q}_{p} \rightarrow \mathbb{Q}_{p} $\, by $ A(x) = \frac{1}{p}x $. Then, $ A $\, is an expansive 
automorphism with modulus, $ \lvert A \rvert = \mu(A(B_{0}(0))) = \mu(B_{1}(0)) = p $\, and $ A^{*} = A $.
\begin{definition}\cite{aes2} 
   A collection of closed spaces $ V_{j} \subset L_{2}(\mathbb{Q}_{p}), \, j \in \mathbb{Z} $, is called a \MRA (MRA) in 
   $ L_{2}(\mathbb{Q}_{p}) $\, if the following axioms hold: 
   \begin{enumerate}
    \item $ V_{j} \subset V_{j+1} $\, for all $ j \in \mathbb{Z} $;
    \item $ \bigcup_{j=-\infty}^{\infty} V_{j} $\, is dense in $ L_{2}(\mathbb{Q}_{p}) $;
    \item $ \bigcap_{j=-\infty}^{\infty} V_{j} = \{ 0 \} $;
    \item $ f(\cdot) \in V_{j} \Leftrightarrow f(A\cdot) \in V_{j+1} $\, for all $ j \in \mathbb{Z} $;
    \item there exists a function $ \phi \in V_{0} $\, such that $ \{\phi(\cdot -a),\, a \in I_{p}\} $\, is an orthonormal 
    basis for $ V_{0} $.
   \end{enumerate} 
  \end{definition}
 The function $ \phi $\, from axiom $ (5) $\, is called scaling function. Then we says that a MRA is generated by its scaling function $ \phi $\, (or $ \phi $\, generates the MRA). It follows immediately from axioms $ (4) $\, and $ (5) $\, that the functions $ p^{j/2}\phi(p^{-j}\cdot -a),\, a \in I_{p} $, form an orthonormal basis for $ V_{j}, \, j \in \mathbb{Z} $. Let $ \phi $\, be an orthogonal scaling function for a MRA $ \{ V_{j}\}_{j \in \mathbb{Z}} $, then 
\begin{equation}\label{eq1.4}
 \phi(x) = \sum_{a \in I_{p}} \alpha(a) \phi(p^{-1}x -a), \quad \alpha(a) \in \mathbb{C}.
\end{equation}
Such equations are called refinement equations, and their solutions are called refinable functions. \par 
Generally, a refinement equation \eqref{eq1.4} does not imply the inclusion property $ V_{0} \subset V_{1} $\, because the set 
of shifts $ I_{p} $\, does not form a group. Indeed, we need all the functions $ \phi(\cdot - b), b \in I_{p} $, to belong to 
the space $ V_{1} $, i.e., the identities $ \phi(x - b) = \sum_{a\in I_{p}} \alpha(a,b) \phi(p^{-1}x - a) $\, should be 
fulfilled for all $ b \in I_{p} $. Since $ p^{-1} b + a $\, is not in $ I_{p} $\, in general, we cannot argue that 
$ \phi(x - b) $\, belongs to $ V_{1} $\, for all $ b \in I_{p} $. Thus the refinement equation must be redefined as in the following definition.
\begin{definition}\cite{aes1,aes2}
 If $ \phi \in L_{2}(\mathbb{Q}_{p}) $\, is a refinable function and $ \text{supp}(\phi) \subset B_{N}(0) $, $ N \geq 0 $, then its 
 refinement equation is
 \begin{equation}\label{eq1.6}
  \phi(x) = \sum_{k=0}^{p^{N+1}-1} h\left(\frac{k}{p^{N+1}}\right) \phi \left( \frac{x}{p}-\frac{k}{p^{N+1}}\right), \quad \forall x \in \mathbb{Q}_{p},
 \end{equation}
with $ h \in l_{0}(I_{p}) $\, such that $ h(a) = 0 $\, for all $ a \in I_{p} \setminus \{k/p^{N+1}: k=0,1,\ldots, p^{N+1}-1 \} $ and 
\begin{equation}\label{eq1.7}
 \sum_{k=0}^{p^{N+1}-1} h\left(\frac{k}{p^{N+1}}\right) = p.
\end{equation}
\end{definition}
Taking Fourier transform on both sides of \eqref{eq1.6}, we obtain
\begin{equation}\label{eq1.8}
 \widehat{\phi}(\xi) = H(A^{-1}\xi)\widehat{\phi}(A^{-1}\xi), \quad \xi \in \mathbb{Q}_{p},
\end{equation}
where
\begin{equation}\label{eq1.9}
 H(\xi) = \frac{1}{p}\sum_{k=0}^{p^{N+1}-1} h\left(\frac{k}{p^{N+1}}\right)\chi\left(\frac{k}{p^{N+1}}, \xi\right), \quad \xi \in \mathbb{Q}_{p}.
\end{equation}

Denote by $ \mathscr{D}_{N}^{M} $\, the set of all $ p^{M} $-periodic functions supported on $ B_{N}(0) $. Then we have the following results about refinable functions \cite{aes2}.
\begin{theorem}\label{th1.1}
 A function $ \phi \in \mathscr{D}_{N}^{M}, \, M,N \geq 0 $, with $ \widehat{\phi}(0) \neq 0 $\, generates a MRA if and only if
 \begin{enumerate}
  \item $ \phi $\, is refinable;
  \item there exists at most $ p^{N} $\, integers $ l $\, such that $ 0 \leq l < p^{M+N} $\, and 
        $ \widehat{\phi}(\frac{l}{p^{M}}) \neq 0 $.
 \end{enumerate}
\end{theorem}
\begin{theorem}\label{th1.2}
 Let $ \widehat{\phi} $\, be defined by \eqref{eq1.8}, where $ H $\, is the trigonometric polynomial \eqref{eq1.9} with 
 $ H(0) = 1 $. If $ H(k) = 0 $\, for all $ k = 1,\ldots,p^{N+1}-1 $\, not divisible by $ p $, then $ \phi \in \mathscr{D}_{N}^{0} $. 
 If furthermore, $ \lvert H(k) \rvert = 1 $\, for all $ k = 1,\ldots,p^{N+1}-1 $\, divisible by $ p $, then 
 $ \{\phi(x -a),\, a \in I_{p}\} $\, is an orthonormal system. Conversely, if $ \text{supp }\widehat{\phi} \subset B_{0}(0) $\, 
 and the system $ \{\phi(x -a),\, a \in I_{p}\} $\, is orthonormal, then $ \lvert H(k) \rvert = 0 $\, whenever $ k $\, is 
 not divisible by $ p $, $ \lvert H(k) \rvert = 1 $\, whenever $ k $\, is divisible by $ p $, $ k = 1,\ldots,p^{N+1}-1 $, and 
 $ \lvert \widehat{\phi}(x) \rvert = 1 $\, for any $ x \in B_{0}(0) $.
\end{theorem}

Suppose we have a p-adic MRA generated using a scaling function $ \phi $\, satisfying the refinement equation \eqref{eq1.6}, then 
the wavelet functions $ \psi_{j}, \, j = 1,\ldots, p-1 $\, are in the form
\begin{equation}\label{eq1.10}
 \psi_{j}(x) = \sum_{k=0}^{p^{N+1}-1} h_{j}\left(\frac{k}{p^{N+1}}\right) \phi \left( \frac{x}{p}-\frac{k}{p^{N+1}}\right), \quad 
 \forall x \in \mathbb{Q}_{p},
\end{equation}
where the coefficients $ h_{j}(\frac{k}{p^{N+1}}) $\, are chosen such that
\begin{equation}\label{eq1.11}
 <\psi_{j},\phi(\cdot-a)> = 0, \, \, <\psi_{j}, \psi_{k}(\cdot - a)> = \delta_{j,k} \delta_{0,a}, \,\, j,k = 1,\ldots, p-1, 
\end{equation}
for any $ a \in I_{p} $. \newline 
Set $$ h = \frac{1}{\sqrt{p}}\left(h(0),\ldots,h\left(\frac{p^{N+1}-1}{p^{N+1}}\right)\right),$$ 
$$ h_{j} = \frac{1}{\sqrt{p}}\left(h_{j}(0),\ldots,h_{j}\left(\frac{p^{N+1}-1}{p^{N+1}}\right)\right),\, j = 1,\ldots, p-1, $$
and
$$ S= \begin{bmatrix}
      0 & 0 & \cdots & 0 & 1 \\
      1 & 0 & \cdots & 0 & 0 \\
      0 & 1 & \cdots & 0 & 0 \\
      \cdots & \cdots & \cdots & \cdots \\
      0 & 0 & \cdots & 1 & 0 
     \end{bmatrix}. $$ 
In order to satisfy \eqref{eq1.11}, we need to find $ h_{j}, \, j = 1, \ldots, p-1 $\, such that the matrix
\begin{equation}\label{eq1.12}
 U = (S^{0}h,\ldots,S^{p^{N}-1}h,S^{0}h_{1},\ldots,S^{p^{N}-1}h_{1}, \ldots, S^{0}h_{p-1},\ldots,S^{p^{N}-1}h_{p-1})
\end{equation}
is unitary.\cite{kss2} \par 
In order to solve the refinement equation \eqref{eq1.6}, we start with a compactly supported function $ \phi \in L_{q}(\mathbb{Q}_{p})\, (1 \leq q \leq \infty) $\, and use the iteration scheme $ f_{n} := Q_{h}^{n}\phi,\,
n = 0, 1, 2,\ldots $, where $ Q_{h} $\, is the bounded linear operator on $ L_{q}(\mathbb{Q}_{p}) $\, given by
\begin{equation}\label{eq1.13}
 Q_{h}\phi := \sum_{\alpha \in I_{p}}h(\alpha)\phi(A\cdot-\alpha), \quad \phi \in L_{q}(\mathbb{Q}_{p}).
\end{equation}
This iteration scheme is called a subdivision scheme (\cite{cdm}) which is also referred to as a cascade algorithm.

\section{Subdivision operator and Joint spectral radius}

This section is devoted to the study of subdivision operator associated with refinement mask and joint spectral radii of a finite collection of linear operators. 
\begin{definition}
 The subdivision operator $ S_{h} $\, associated with the refinement mask $ h $\, is the linear operator on $ l(I_{p}) $\, defined by 
 \begin{equation}\label{eq3.1}
  S_{h}u(\alpha) := \sum_{\beta \in I_{p}} h(\alpha - A\beta)u(\beta), \qquad \alpha \in I_{p}, 
 \end{equation}
where $ u \in l(I_{p}) $. 
\end{definition}
Then $ S_{h}\delta(\alpha) = h(\alpha) $. Now, for $ f \in L_{q}(\mathbb{Q}_{p})\, (1\leq q \leq \infty) $\, we have 
$$ Q_{h}f = \sum_{\alpha \in I_{p}} S_{h}\delta(\alpha)f(A\cdot -\alpha). $$ By induction on $ n $, we get 
$$ \begin{aligned}
    Q_{h}^{n} f & = Q_{h}^{n-1}(Q_{h}f) = \sum_{\alpha \in I_{p}} S_{h}^{n-1}\delta(\alpha)Q_{h}f(A^{n-1}\cdot-\alpha) \\
    & = \sum_{\alpha \in I_{p}} S_{h}^{n}\delta(\alpha)\sum_{\beta \in I_{p}}h(\beta)f(A^{n}\cdot -A\alpha-\beta) \\
    & = \sum_{\beta \in I_{p}} [\sum_{\alpha \in I_{p}} h(\beta-A\alpha)S_{h}^{n}\delta(\alpha)]f(A^{n}\cdot -\beta) \\
    & = \sum_{\beta \in I_{p}} S_{h}^{n}\delta(\beta)f(A^{n}\cdot -\beta).
   \end{aligned}
$$ That is 
\begin{equation}\label{eq3.2}
 Q_{h}^{n} f = \sum_{\beta \in I_{p}} S_{h}^{n}\delta(\beta)f(A^{n}\cdot -\beta).
\end{equation}
We can observe that any compactly supported refinable function $ \phi $\, satisfies the relation $ Q_{h}^{n}\phi = \phi $\, for all $ n \in \mathbb{N} $. \par 
Let $ \mathcal{A} $\, be a finite collection of linear operators on a finite dimensional vector space $ V $. A vector norm $ \lVert \cdot \rVert $\, on $ V $\, induces a norm on the linear operators on $ V $\, as follows. For a linear operator $ A $\, on $ V $, define
$$ \lVert A \rVert := \max_{\lVert v \rVert = 1}\{ \lVert Av \rVert \}. $$
For a positive integer $ n $, we denote by $ \mathcal{A}^{n} $\, the Cartesian power of $ \mathcal{A} $:
$$ \mathcal{A}^{n} = \{ (A_{1},\ldots,A_{n}): A_{1},\ldots, A_{n} \in \mathcal{A} \}. $$
When $ n = 0 $, we interpret $ \mathcal{A}^{0} $\, as the set $ \{I\} $, where $ I $\, is the identity mapping on $ V $. Let
$$ \lVert \mathcal{A}^{n} \rVert_{\infty} := \max \{ \lVert A_{1}\cdots A_{n} \rVert : (A_{1},\ldots,A_{n}) \in \mathcal{A}^{n} \}. $$
Then the uniform joint spectral radius of $ \mathcal{A} $, introduced by Rota and Strang \cite{rs}, is defined as
$$ \rho_{\infty}(\mathcal{A}) := \lim_{n \rightarrow \infty} \lVert \mathcal{A}^{n} \rVert_{\infty}^{1/n}. $$
  \par 
The $ q $-norm joint spectral radius of a finite collection of linear operators was introduced by Jia \cite{rqj4}. For $ 1 \leq q < \infty $,
$$ \lVert \mathcal{A}^{n} \rVert_{q} := \left( \sum_{(A_{1},\ldots,A_{n}) \in \mathcal{A}^{n}} \lVert A_{1}\cdots A_{n} \rVert^{q} \right)^{1/q}. $$
For $ 1 \leq q \leq \infty $, the $ q $-norm joint spectral radius of $ \mathcal{A} $\, is defined to be
$$ \rho_{q}(\mathcal{A}) := \lim_{n \rightarrow \infty} \lVert \mathcal{A}^{n} \rVert_{q}^{1/n}. $$
It is easily seen that this limit indeed exists, and
$$ \lim_{n \rightarrow \infty} \lVert \mathcal{A}^{n} \rVert_{q}^{1/n} = \inf_{n \geq 1} \lVert \mathcal{A}^{n} \rVert_{q}^{1/n}. $$
Clearly, $ \rho_{q}(\mathcal{A}) $\, is independent of the choice of the vector norm on $ V $.\par
If $ \mathcal{A} $\, consists of a single linear operator $ A $, then
$$ \rho_{q}(\mathcal{A}) = \rho(A), $$
where $ \rho(A) $\, denotes the spectral radius of $ A $, which is independent of $ q $. If $ \mathcal{A} $\, consists of more than one element, then $ \rho_{q}(\mathcal{A}) $\, depends on $ q $\, in general. By some basic properties of $ l_{q} $\, spaces, we have that, for $ 1 \leq q \leq r \leq \infty $,
$$ (\#\mathcal{A}^{[(1/r)-(1/q)]}\rho_{q}(\mathcal{A}) \leq \rho_{r}(\mathcal{A}) \leq \rho_{q}(\mathcal{A}), $$
where $ \#\mathcal{A} $\, denotes the number of elements in $ \mathcal{A} $. Furthermore, it is easily seen from the definition of the joint spectral radius that $ \rho(A) \leq \rho_{\infty}(\mathcal{A}) $\, for any element $ A $\, in $ \mathcal{A} $. \par
Recall that $ S_{h} $\, is the subdivision operator given in \eqref{eq3.1}. From \eqref{eq3.2} we see that, in order to study convergence of the subdivision scheme, we need to analyze the sequences $ S_{h}^{n}\delta $, $ n = 1, 2, \ldots $ For this purpose, we introduce the biinfinite matrices
$ A_{\epsilon}\, (\epsilon \in I_{p}) $\, as follows:
\begin{equation}\label{eq3.3}
 A_{\epsilon}(\alpha, \beta) := h(\epsilon+A\alpha-\beta), \qquad \alpha, \beta \in I_{p}.
\end{equation}
\begin{lemma}\label{lm3.1}
 Suppose $ \alpha = \epsilon_{1}+A\epsilon_{2}+\cdots+A^{n-1}\epsilon_{n}+A^{n}\gamma $, where $ \epsilon_{1},\epsilon_{2},\ldots,\epsilon_{n},\gamma \in I_{p} $. Then for any $ \beta \in I_{p} $, 
 $$ S_{h}^{n}\delta(\alpha -\beta) = A_{\epsilon_{n}}\cdots A_{\epsilon_{1}}(\gamma,\beta). $$
\end{lemma}
\begin{proof}
 The proof proceeds by induction on $ n $. For $ n = 1 $\, and $ \alpha = \epsilon_{1} + A\gamma $, we have
$$ \begin{aligned}
    S_{h}\delta(\alpha-\beta) & = \sum_{\eta \in I_{p}} h(\alpha-\beta-A\eta)\delta(\eta) = h(\alpha-\beta) \\ & = h(\epsilon_{1}+A\gamma-\beta) = A_{\epsilon_{1}}(\gamma,\beta).  
   \end{aligned} $$
Suppose $ n > 1 $\, and the lemma has been verified for $ n - 1 $. For $ \alpha = \epsilon_{1} + A \alpha_{1} $, where $ \alpha_{1} , \epsilon_{1} \in I_{p} $, we have
\begin{equation}\label{eq3.4}
 \begin{aligned}
  S_{h}^{n}\delta(\alpha-\beta) & = \sum_{\eta \in I_{p}} h(\alpha-\beta-A\eta)S_{h}^{n-1}\delta(\eta) \\ & = \sum_{\eta \in I_{p}} h(\epsilon_{1}+A\alpha_{1}-\beta-A\eta)S_{h}^{n-1}\delta(\eta) \\ & = \sum_{\eta \in I_{p}} h(\epsilon_{1}+A\eta-\beta)S_{h}^{n-1}\delta(\alpha_{1}-\eta).
 \end{aligned}
\end{equation}
Suppose $ \alpha_{1} = \epsilon_{2}+A\epsilon_{3}+\cdots+A^{n-2}\epsilon_{n}+A^{n-1}\gamma $. Then by the induction hypothesis we have
$$ S_{h}^{n-1}\delta(\alpha_{1} -\eta) = A_{\epsilon_{n}}\cdots A_{\epsilon_{2}} (\gamma,\eta). $$
This, in connection with \eqref{eq3.4}, gives
$$ S_{h}^{n}\delta(\alpha -\beta) = \sum_{\eta \in I_{p}} A_{\epsilon_{1}}(\eta,\beta) A_{\epsilon_{n}}\cdots A_{\epsilon_{2}} (\gamma,\eta) = A_{\epsilon_{n}}\cdots A_{\epsilon_{1}}(\gamma,\beta), $$
thereby completing the induction procedure.
\end{proof}
The biinfinite matrices $ A_{\epsilon}\, (\epsilon \in I_{p}) $\, defined in \eqref{eq3.3} may be viewed as the linear operators on $ l_{0}(I_{p}) $\, given by
\begin{equation}\label{eq3.5}
 A_{\epsilon}v(\alpha) = \sum_{\beta \in I_{p}} h(\epsilon+A\alpha-\beta)v(\beta), \qquad v \in l_{0}(I_{p}),\, \alpha \in I_{p}.
\end{equation} \par 
Now let $ \mathcal{A} $\, be a finite collection of linear operators on a vector space $ V $, which is not necessarily finite dimensional. A subspace $ W $\, of $ V $\, is said to be $ \mathcal{A} $-invariant if it is invariant under every operator $ A $\, in $ \mathcal{A} $. Let $ U $\, be a subset of $ V $. The intersection of all $ \mathcal{A} $-invariant subspaces of $ V $\, containing $ U $\, is $ \mathcal{A} $-invariant, and we call it the minimal $ \mathcal{A} $-invariant subspace generated by $ U $, or the minimal common invariant subspace of the operators $ A $\, in $ \mathcal{A} $\, generated by $ U $. This subspace is spanned by the set
$$ \{A_{1}\cdots A_{j}u : u \in U, (A_{1},\ldots,A_{j}) \in \mathcal{A}^{j},\, j=0,1,\ldots\}. $$
If, in addition, $ V $\, is finite dimensional, then there exists a positive integer $ k $\, such that the set
$$ \{A_{1}\cdots A_{j}u : u \in U, (A_{1},\ldots,A_{j}) \in \mathcal{A}^{j},\, j=0,1,\ldots, k \}, $$
already spans the minimal $ \mathcal{A} $-invariant subspace generated by $ U $. \par 
We define, for $ 1 \leq q < \infty $,
$$ \lVert \mathcal{A}^{n}v \rVert_{q} := \left( \sum_{(A_{1},\ldots,A_{n}) \in \mathcal{A}^{n}} \lVert A_{1}\cdots A_{n}v \rVert^{q} \right)^{1/q}. $$
and, for $ q = \infty $,
$$ \lVert \mathcal{A}^{n}v \rVert_{\infty} := \max \{ \lVert A_{1}\cdots A_{n}v \rVert : (A_{1},\ldots,A_{n}) \in \mathcal{A}^{n} \}. $$ \par 
The symbol of a sequence $ a \in l_{0}(I_{p}) $\, is the Laurent polynomial $ \tilde{a}(z) $\, given by
$$ \tilde{a}(z) := \sum_{\alpha \in I_{p}}a(\alpha)z^{\alpha}, \qquad z \in \mathbb{C}\setminus\{0\}. $$ \par 
For $ \beta \in I_{p} $, we denote by $ \tau^{\beta} $\, the shift operator on $ l_{0}(I_{p}) $\, given by
$$ \tau^{\beta}\lambda := \lambda(\cdot - \beta) \text{   for } \lambda \in l_{0}(I_{p}). $$
Let $ v $\, be an element of $ l_{0}(I_{p}) $. Then its symbol $ \tilde{v}(z) $\, is a Laurent polynomial, which induces the difference operator $ \tilde{v}(\tau) := \sum_{\beta \in I_{p}} v(\beta)\tau^{\beta} $. \par 
Let $ E $\, be a complete set of representatives of the distinct cosets of the quotient group $ I_{p}/AI_{p} $. We assume that $ E $\, contains $ 0 $. Thus, each element $ \alpha \in I_{p} $\, can be uniquely represented as $ \epsilon +A\gamma $, where $ \epsilon \in E $\, and $ \gamma \in I_{p} $. \par 
As usual, for $ 1 \leq q \leq \infty $, $ l_{q}(I_{p}) $\, denotes the Banach space of all sequences $ a $\, on $ I_{p} $\, such that $ \lVert a \rVert_{q} < \infty $, where
$$ \lVert a \rVert_{q} := \left(\sum_{\alpha \in I_{p}} \lvert a(\alpha) \rvert^{q}\right)^{1/q} \qquad \text{for } 1\leq q < \infty, $$
and $ \lVert a \rVert_{\infty} $\, is the supremum of $ a $\, on $ I_{p} $. In the following lemma, the underlying vector norm on $ l_{0}(I_{p}) $\, is chosen to be the $ l_{q} $-norm.
\begin{lemma}\label{lm3.2}
 Let $ S_{h} $\, be the subdivision operator associated with an expansive automorphism $ A $\, and a mask $ h $. Let $ \mathcal{A} := \{A_{\epsilon} : \epsilon \in E \} $, where $ A_{\epsilon} $\, are the linear operators on $ l_{0}(I_{p}) $\, given by \eqref{eq3.5}. Then for $ 1 \leq q \leq \infty $\, and $ v \in l_{0}(I_{p}) $,
 \begin{equation}\label{eq3.6}
  \lVert \tilde{v}(\tau)S_{h}^{n}\delta \rVert_{q} = \lVert \mathcal{A}^{n}v \rVert_{q}, \qquad n=1,2,\ldots .
 \end{equation}
\end{lemma}
\begin{proof}
 Suppose $ \alpha = \epsilon_{1}+A\epsilon_{2}+\cdots+A^{n-1}\epsilon_{n}+A^{n}\gamma $, where $ \epsilon_{1},\epsilon_{2},\ldots,\epsilon_{n} \in E $\, and $ \gamma \in I_{p} $. Then by Lemma \ref{lm3.1} we have
 $$ \begin{aligned}
     \tilde{v}(\tau)S_{h}^{n}\delta(\alpha) & = \sum_{\beta \in I_{p}} v(\beta) \tau^{\beta} S_{h}^{n}(\alpha) = \sum_{\beta \in I_{p}} v(\beta) S_{h}^{n}(\alpha-\beta) \\ & = \sum_{\beta \in I_{p}} A_{\epsilon_{n}}\cdots A_{\epsilon_{1}}(\gamma,\beta) v(\beta) = A_{\epsilon_{n}}\cdots A_{\epsilon_{1}}v(\gamma)
    \end{aligned} $$
Hence, \eqref{eq3.6} is true for $ q = \infty $. When $ 1 \leq q < \infty $\, we have
$$ \sum_{\alpha \in I_{p}} \lvert \tilde{v}(\tau)S_{h}^{n}\delta(\alpha) \rvert^{q} = \sum_{(\epsilon_{1},\ldots,\epsilon_{n})\in E^{n}} \sum_{\gamma \in I_{p}} \lvert A_{\epsilon_{n}}\cdots A_{\epsilon_{1}}v(\gamma)\rvert^{q}. $$
This verifies \eqref{eq3.6} for $ 1 \leq q < \infty $.
\end{proof}
Let $ \mathcal{A} := \{A_{\epsilon} : \epsilon \in E \} $. We claim that, for each $ v \in l_{0}(I_{p}) $, the minimal $ \mathcal{A} $-invariant subspace generated by $ v $\, is finite dimensional. To establish this result, we shall introduce the concept of admissible sets. For a finite subset $ K $\, of $ I_{p} $, recall that $ l(K) $\, is the linear subspace of $ l_{0}(I_{p}) $\, consisting of all sequences supported on $ K $. Let $ A $\, be a linear operator on $ l_{0}(I_{p}) $. A finite subset $ K $\, of $ I_{p} $\, is said to be admissible for $ A $\, if $ l(K) $\, is invariant under $ A $. The
following lemma shows that there exists a finite subset $ K $\, of $ I_{p} $\, such that $ K $\, contains the support of $ v $\, and is admissible for all $ A_{\epsilon},\, \epsilon \in E $.
\begin{lemma}\label{lm3.3}
 Suppose $ A $\, is an expansive automorphism and $ h $\, is a sequence on $ I_{p} $\, with its support $ \Omega := \{\alpha \in I_{p} : h(\alpha) \neq 0 \} $\, being finite. Let $ A_{\epsilon}\, (\epsilon \in E) $\, be the linear operators on $ l_{0}(I_{p}) $\, given by \eqref{eq3.3}. Then a finite subset $ K $\, of $ I_{p} $\, is admissible for $ A_{0} $\, if and only if
\begin{equation}\label{eq3.7}
 A^{-1}(\Omega+K)\cap I_{p} \subseteq K.
\end{equation}
Consequently, for any finite subset $ G $\, of $ I_{p} $, there exists a finite subset $ K $\, of $ I_{p} $\, such that $ K $\, contains $ G $\, and $ K $\, is admissible for all $ A_{\epsilon},\, \epsilon \in E $.
\end{lemma}
\begin{proof}
 Suppose $ K $\, is admissible for $ A_{0} $. Let $ \alpha \in A^{-1}(\Omega+K)\cap I_{p} $. Then we have $ A\alpha = \gamma + \beta $\, for some $ \gamma \in \Omega $\, and $ \beta \in K $. It follows that
$$ A_{0}\delta_{\beta}(\alpha) = h(A\alpha-\beta) = h(\gamma) \neq 0. $$
Since $ K $\, is admissible for $ A_{0} $, we have $ A_{0}\delta_{\beta} \in l(K) $, and therefore $ \alpha \in K $. This shows that \eqref{eq3.7} is true. \par 
Conversely, suppose \eqref{eq3.7} is true. Let $ v \in l(K) $\, and $ \alpha \in I_{p} $. Then 
$$ A_{0}v(\alpha) = \sum_{\beta \in I_{p}}h(A\alpha-\beta)v(\beta) \neq 0 $$
implies that $ A\alpha-\beta \in \Omega $\, for some $ \beta \in K $. It follows that $ A\alpha \in \Omega + K $. Therefore, $ \alpha \in A^{-1}(\Omega+K)\cap I_{p} $, and so $ \alpha \in K $\, by \eqref{eq3.7}. This shows that $ A_{0} $\, maps $ l(K) $\, to $ l(K) $. In other words, $ K $\, is admissible for $ A_{0} $. \par 
From the above proof we see that a finite subset $ K $\, of $ I_{p} $\, is admissible for $ A_{\epsilon} $\, if and only if 
\begin{equation}\label{eq3.8}
 A^{-1}(\Omega-\epsilon+K)\cap I_{p} \subseteq K.
\end{equation}
The set $ \Omega - E $\, consists of all the points $ \omega - \epsilon $, where $ \omega \in \Omega $\, and $ \epsilon \in E $. \par 
Now suppose $ G $\, is a finite subset of $ I_{p} $. Let $ B := A G \cup (\Omega - E) \cup \{0\} $, and let 
$$ K:=\left( \sum_{n=1}^{\infty} A^{-n}B \right) \cap I_{p}. $$
In other words, an element $ \alpha \in I_{p} $\, belongs to $ K $\, if and only if $ \alpha = \sum_{n=1}^{\infty} A^{-n}b_{n} $\, for some sequence of elements $ b_{n} \in B $. Since $ 0 \in B $\, and $ A^{-1}B \supseteq G $, we have 
$$ K \supseteq I_{p}\cap A^{-1}B \supseteq I_{p}\cap G = G. $$
Moreover,
$$ \begin{aligned}
    A^{-1}(\Omega-\epsilon+K )\cap I_{p} & \subseteq A^{-1}(B+K)\cap I_{p}\\ & = (A^{-1}B+A^{-1}K) \cap I_{p} \\ & \subseteq (A^{-1}B+A^{-2}B+\cdots )\cap I_{p} = K.
   \end{aligned}
 $$
Thus, $ K $\, satisfies \eqref{eq3.8}. Hence, $ K $\, is admissible for all $ A_{\epsilon},\, \epsilon \in E $.
\end{proof}
\begin{lemma}\label{lm3.4}
 Let $ \mathcal{A} $\, be a finite collection of linear operators on a vector space $ V $. Let $ v $\, be a vector in $ V$, and let $ V(v) $\, be the minimal $ \mathcal{A} $-invariant subspace generated by $ v $. If $ V(v) $\, is finite dimensional, then
\begin{equation}\label{eq3.9}
 \lim_{n \rightarrow \infty} \lVert \mathcal{A}^{n}v \rVert_{q}^{1/n} = \rho_{q}(\mathcal{A}_{\lvert V(v)}).
\end{equation}
\end{lemma}
\begin{proof}
 Let $ \lVert \cdot \rVert $\, be a vector norm on $ V(v) $. Since $ V(v) $\, is finite dimensional, there exists a positive integer $ k $\, such that $ V(v) $\, is spanned by the set
$$ Y := \{A_{1} \cdots A_{j}v : (A_{1},\ldots, A_{j}) \in \mathcal{A}^{j},\, j = 0, 1,\ldots, k\}. $$
Thus, there exists a constant $ C_{1} > 0 $\, such that $ \lVert \mathcal{A}^{n}y \rVert_{q} \leq C_{1} \lVert \mathcal{A}^{n}v \rVert_{q} $\, for all $ y \in Y $\, and all $ n = 1, 2,\ldots $. Moreover, there exists a positive constant $ C_{2} $\, such that
$$ \lVert \mathcal{A}^{n}_{\lvert V(v)} \rVert_{q} \leq C_{2} \max_{y \in Y} \lVert \mathcal{A}^{n}y \rVert_{q}, \qquad n = 1, 2,\ldots .$$
Therefore, there exists a positive constant $ C $\, such that for all $ n = 1, 2, \ldots $,
$$ \lVert \mathcal{A}^{n}_{\lvert V(v)} \rVert_{q} \leq C \lVert \mathcal{A}^{n}v \rVert_{q}. $$
But $ \lVert \mathcal{A}^{n}v \rVert_{q} \leq \lVert \mathcal{A}^{n}_{\lvert V(v)} \rVert_{q} \lVert v \rVert $. This proves the desired relation \eqref{eq3.9}.
\end{proof}
\begin{theorem}\label{th3.5}
 Let $ S_{h} $\, be the subdivision operator associated with an expansive automorphism $ A $\, and a mask $ h $. Let $ \mathcal{A} := \{A_{\epsilon} : \epsilon \in E \} $, where $ E $\, is a complete set of representatives of the distinct cosets of the quotient group $ I_{p}/AI_{p} $\, and $ A_{\epsilon} $\, are the linear operators on $ l_{0}(I_{p}) $\, given by \eqref{eq3.3}. Then for $ v \in l_{0}(I_{p}) $,
\begin{equation}\label{eq3.10}
 \lim_{n \rightarrow \infty}\lVert \tilde{v}(\tau)S_{h}^{n}\delta \rVert_{q}^{1/n} = \rho_{q}(\{A_{\epsilon \lvert V(v)} : \epsilon \in E \}),
\end{equation}
where $ V(v) $\, is the minimal $ \mathcal{A} $-invariant subspace generated by $ v $. Moreover, if $ W $\, is the minimal $ \mathcal{A} $-invariant subspace generated by a finite set $ Y $, then
\begin{equation}\label{eq3.11}
 \rho_{q}(\{A_{\epsilon \lvert W} : \epsilon \in E \}) = \max_{v \in Y}\{\lim_{n \rightarrow \infty}\lVert \tilde{v}(\tau)S_{h}^{n}\delta \rVert_{q}^{1/n}\}.
\end{equation}
\end{theorem}
\begin{proof}
 By Lemma \ref{lm3.3}, $ V(v) $\, is finite dimensional, and so the relevant joint spectral radius in \eqref{eq3.10} is well defined. By Lemma \ref{lm3.2} we have,
$$ \lVert \tilde{v}(\tau)S_{h}^{n}\delta \rVert_{q} = \lVert \mathcal{A}^{n}v \rVert_{q}, \qquad 1 \leq q \leq \infty,\, n=1,2,\ldots . $$
Applying Lemma \ref{lm3.4} to the present situation, we obtain \eqref{eq3.10}. \par 
For the second part of the theorem, let $ W $\, be the minimal $ \mathcal{A} $-invariant subspace generated by a finite set $ Y $, and observe that $ W $\, is a finite sum of the linear subspaces $ V(v),\, v \in Y $. Hence
$$ \rho_{q}(\{A_{\epsilon \lvert W} : \epsilon \in E \}) = \max_{v \in Y}\{ \rho_{q}(\{A_{\epsilon \lvert V(v)} : \epsilon \in E \})\} $$
This, together with \eqref{eq3.10}, verifies \eqref{eq3.11}.
\end{proof}

\section{Convergence of subdivision schemes}

In this section, we characterize the $ L_{q} $-convergence $ (1 \leq q \leq \infty) $\, of a subdivision scheme in terms of the corresponding refinement mask. \par 
In our study of convergence, the concept of stability plays an important role. The shifts of a function $ \phi $\, in $ L_{q}(\mathbb{Q}_{p}) $\, are said to be stable if there are two positive constants $ C_{1} $\, and $ C_{2} $\, such that
\begin{equation}\label{eq4.1}
 C_{1} \lVert \lambda \rVert_{q} \leq \left \lVert \sum_{\alpha \in I_{p}} \lambda(\alpha)\phi(\cdot -\alpha) \right \rVert_{q} \leq C_{2}\lVert \lambda \rVert_{q}, \qquad \forall \lambda \in l_{0}(I_{p}).
\end{equation}
\begin{remark}
If $ \phi $\, is orthonormal then the shifts of $ \phi $\, are stable.
\end{remark}
\par 
For a function $ \phi $\, and sequence $ a$, by $ supp(\phi) $\, we denote the support of $ \phi $, and by $ supp(h) $\, we denote the set $ \{\alpha \in I_{p} : a(\alpha) \neq 0 \} $. \par
First, we give a necessary condition for the subdivision scheme to converge.
\begin{theorem}\label{th4.1}
 Let $ A $\, be an expansive automorphism with $ \lvert A \rvert = p $, $ h $\, an element in $ l_{0}(I_{p}) $\, with $ \sum_{\alpha \in I_{p}} h(\alpha) = p $, and $ S_{h} $\, the subdivision operator associated with $ A $\, and $ h $. If the subdivision scheme associated with $ A $\, and $ h $\, converges to a compactly supported orthogonal function in the $ L_{q} $-norm $ (1 \leq q \leq \infty) $, then for any vector $ y \in supp(h) $,
\begin{equation}\label{eq4.2}
 \lim_{n \rightarrow \infty} p^{-n/q}\lVert \nabla_{y}S_{h}^{n}\delta \rVert_{q} = 0.
\end{equation}
Consequently, if the subdivision scheme associated with $ h $\, converges to a compactly supported orthogonal function in the $ L_{q} $-norm, then
\begin{equation}\label{eq4.3}
 \sum_{\beta \in I_{p}}h(\alpha-A\beta) = 1, \qquad \forall \alpha \in I_{p}.
\end{equation}
\end{theorem}
\begin{proof}
 Suppose $ \phi $\, is a compactly supported orthogonal function in $ L_{q}(\mathbb{Q}_{p}) $. Then the shifts of $ \phi $\, are stable. For $ n = 0, 1, 2, \ldots $, let $ a_{n} := S_{h}^{n}\delta $\, and $ f_{n} := Q_{h}^{n}\phi $, where $ Q_{h} $\, is the operator given in \eqref{eq1.13}. Then by \eqref{eq3.1} we have,
$$ f_{n} = \sum_{\alpha \in I_{p}} a_{n}(\alpha) \phi(A^{n}\cdot -\alpha). $$
Hence, for $ y \in supp(h) $\, we have
$$ \begin{aligned}
    f_{n}-f_{n}(\cdot-A^{-n}y) & = \sum_{\alpha \in I_{p}} a_{n}(\alpha) \phi(A^{n}\cdot -\alpha) - \sum_{\alpha \in I_{p}} a_{n}(\alpha) \phi(A^{n}\cdot -y -\alpha) \\ & = \sum_{\alpha \in I_{p}} a_{n}(\alpha) \phi(A^{n}\cdot -\alpha)- \sum_{\alpha \in I_{p}} a_{n}(\alpha-y) \phi(A^{n}\cdot -\alpha) \\ & = \sum_{\alpha \in I_{p}} [a_{n}(\alpha)-a_{n}(\alpha-y)] \phi(A^{n}\cdot -\alpha) \\
    & = \sum_{\alpha \in I_{p}} \nabla_{y}a_{n}(\alpha) \phi(A^{n}\cdot -\alpha).
   \end{aligned}
$$
Since the shifts of $ \phi $\, are stable, there exists a constant $ C > 0 $\, such that
\begin{equation}\label{eq4.4}
 p^{-n/q}\lVert \nabla_{y}a_{n} \rVert_{q} \leq C \lVert f_{n}-f_{n}(\cdot-A^{-n}y) \rVert_{q}.
\end{equation}
If the subdivision scheme converges in the $ L_{q} $-norm, then there exists a compactly supported orthogonal function $ f $\, in $ L_{q}(\mathbb{Q}_{p})$\, ($ f \in C(\mathbb{Q}_{p})$\, in the case $ q = \infty$) such that $ \lVert f_{n} - f \rVert_{q} \rightarrow 0 $\, as $ n \rightarrow \infty $. Moreover, by the triangle inequality, we have
$$ \lVert f_{n}-f_{n}(\cdot-A^{-n}y) \rVert_{q} \leq \lVert f-f(\cdot-A^{-n}y)\rVert_{q}+2\lVert f-f_{n} \rVert_{q}. $$
Hence, $ \lVert f_{n}-f_{n}(\cdot-A^{-n}y) \rVert_{q} \rightarrow 0 $\, as $ n \rightarrow \infty $. This, together with \eqref{eq4.4}, verifies \eqref{eq4.2}. \par 
For the second part of the theorem, we observe that if the subdivision scheme converges in the $ L_{q} $-norm for some $ q $\, with $ 1 \leq q \leq \infty $, then it also converges in the $ L_{1} $-norm. \par 
Let $ E $\, be a complete set of representatives of the distinct cosets of $ I_{p}/AI_{p} $. Then $ \#E = p $, and $ I_{p} $\, is the disjoint union of $ \alpha + AI_{p},\, \alpha \in E $. Since $ \sum_{\alpha \in I_{p}} h(\alpha) = p $, we have
$$ \sum_{\alpha \in E} \sum_{\beta \in I_{p}} h(\alpha-A\beta) = p. $$
Thus, \eqref{eq4.3} will be proved if we can show that
\begin{equation}\label{eq4.5}
 \sum_{\beta \in I_{p}} h(\alpha-A\beta)= \sum_{\beta \in I_{p}} h(-A\beta) , \quad \forall \alpha \in E.
\end{equation}
To this end, we deduce from $ a_{n} = S_{h}a_{n-1} $\, that
$$ \sum_{\alpha \in I_{p}} a_{n}(\alpha)= \sum_{\alpha \in I_{p}}\sum_{\beta \in I_{p}} h(\alpha-A\beta)a_{n-1}(\beta) = p \sum_{\beta \in I_{p}} a_{n-1}(\beta). $$
An induction argument gives $ \sum_{\alpha \in I_{p}} a_{n}(\alpha) = p^{n} $. Moreover,
$$ \begin{aligned}
    \sum_{\beta \in I_{p}} a_{n}(\alpha-A\beta) & = \sum_{\beta \in I_{p}}\sum_{\gamma \in I_{p}} h(\alpha-A\beta-A\gamma)a_{n-1}(\gamma) \\ & = \sum_{\beta \in I_{p}} h(\alpha-A\beta) \sum_{\gamma \in I_{p}} a_{n-1}(\gamma-\beta) \\ & = p^{n-1} \sum_{\beta \in I_{p}} h(\alpha-A\beta). 
   \end{aligned}
 $$
Thus we have
$$ \sum_{\beta \in I_{p}} [h(\alpha-A\beta) - h(-A\beta)] = p^{-(n-1)} \sum_{\beta \in I_{p}} [a_{n}(\alpha-A\beta) - a_{n}(-A\beta)] $$
It follows that 
\begin{equation}\label{eq4.6}
 \left \lvert \sum_{\beta \in I_{p}} [h(\alpha-A\beta) - h(-A\beta)] \right \rvert \leq p^{-(n-1)}\lVert \nabla_{\alpha}a_{n} \rVert_{1}.
\end{equation}
If the subdivision scheme is $ L_{1} $-convergent, then by the first part of the theorem we have $ p^{-(n-1)}\lVert \nabla_{\alpha}a_{n} \rVert_{1} \rightarrow 0 $\, as $ n \rightarrow \infty $. This, together with \eqref{eq4.6}, implies \eqref{eq4.5}, as desired.
\end{proof}
\begin{corollary}
     Let $ T $\, be an expansive automorphism with $ \lvert T \rvert = p $, $ h $\, an element in $ l_{0}(I_{p}) $\, with $ \sum_{\alpha \in I_{p}} h(\alpha) = p $, and $ S_{h} $\, the subdivision operator associated with $ A $\, and $ h $. If the subdivision scheme associated with $ T $\, and $ h $\, converges to a compactly supported function with stable shifts in the $ L_{q} $-norm $ (1 \leq q \leq \infty) $, then for any vector $ y \in supp(h) $, Eq. \eqref{eq4.2} holds.
\end{corollary}
\begin{remark}
    If $ \phi $\, is a refinable function that generates a \MRA, then $ \phi $\, has stable translates.
\end{remark}
\par 
The next theorem gives a characterization of convergence of the subdivision
scheme.
\begin{theorem}\label{th4.2}
 Let $ A $\, be an expansive automorphism with $ \lvert A \rvert = p $, $ h $\, an element in $ l_{0}(I_{p}) $\, such that $ \sum_{\alpha \in I_{p}} h(\alpha) = p $, and $ S_{h} $\, the corresponding subdivision operator. Then the subdivision scheme associated with $ A $\, and $ h $\, converges to a compactly supported orthogonal function in the $ L_{q} $-norm $ (1 \leq q \leq \infty) $\, if and only if
\begin{equation}\label{eq4.7}
 \lim_{n \rightarrow \infty} \lVert \nabla_{\gamma}S_{h}^{n}\delta \rVert_{q} < p^{1/q}, \qquad \text{for } \gamma \in supp(h)\setminus \{0\}.
\end{equation}
\end{theorem}
\begin{proof}
 Let $ A_{\epsilon} $\, be the linear operators on $ l_{0}(I_{p}) $\, given by \eqref{eq3.3}, and let $ V $\, be the minimal common invariant subspace of $ A_{\epsilon}\,(\epsilon \in E)$\, generated by $ \nabla_{\gamma}\delta, \, \gamma \in supp(h)\setminus\{0\} $. Then $ V $\, is finite dimensional, and by Theorem \ref{th3.5} we have
$$ \rho_{q} := \rho_{q}(\{A_{\epsilon \lvert V} : \epsilon \in E \}) = \max_{\gamma \in supp(h)\setminus \{0\}} \{ \lim_{n \rightarrow \infty} \lVert \nabla_{\gamma}S_{h}^{n}\delta \rVert_{q}^{1/n}\}. $$
Thus, \eqref{eq4.7} is equivalent to $ \rho_{q}(\{A_{\epsilon \lvert V} : \epsilon \in E \})<p^{1/q} $. \par 
Let $ \mathcal{A} := \{A_{\epsilon \lvert V} : \epsilon \in E\} $. If $ \rho_{q}(\{A_{\epsilon \lvert V} : \epsilon \in E \}) \geq p^{1/q} $, then we have
$$ \inf_{n\geq 1}\lVert \mathcal{A}^{n}\rVert_{q}^{1/n} = \lim_{n \rightarrow \infty} \lVert \mathcal{A}^{n}\rVert_{q}^{1/n} \geq p^{1/q}. $$
It follows that
$$ p^{-n/q}\lVert \mathcal{A}^{n}\rVert_{q} \geq 1, \qquad \forall n. $$
From the proof of Lemma \ref{lm3.4}, we see that there exists a positive constant $ C $\, such that $$ \lVert \mathcal{A}^{n}\rVert_{q} \leq C \max_{\gamma \in supp(h)\setminus \{0\}} \lVert \mathcal{A}^{n}\nabla_{\gamma} \delta \rVert_{q}, \text{ for all } n. $$ Moreover, by Lemma \ref{lm3.2}, we have $ \lVert \mathcal{A}^{n}\nabla_{\gamma} \delta \rVert_{q} = \lVert \nabla_{\gamma}S_{h}^{n}\delta \rVert_{q} $. Hence,
$$ \rho_{q} \geq p^{1/q} \Rightarrow \max_{\gamma \in supp(h)\setminus \{0\}}\{p^{-n/q}\lVert \nabla_{\gamma}S_{h}^{n}\delta \rVert_{q} \} \geq 1/C. $$
Thus, the subdivision scheme associated with $ h $\, is not $ L_{q} $-convergent, by Theorem \ref{th4.1}. This shows that \eqref{eq4.7} is necessary for the subdivision scheme to converge in the $ L_{q} $-norm.\par 
In order to prove the sufficiency part of the theorem, we pick a compactly supported orthogonal function $ \phi $\, in $ L_{q}(\mathbb{Q}_{p}) $\, such that $ \sum_{\alpha \in I_{p}}\phi(\cdot -\alpha) = 1 $. (In the case $ q = \infty $, we assume that $ \phi $\, is continuous.) Let $ f_{n} := Q_{h}^{n}\phi_{0} $\, and $ g_{n} := Q_{h}^{n}\phi $\,
for $ n = 1, 2, \ldots $, where $ Q_{h} $\, is the operator given in \eqref{eq1.13} and $ \phi_{0} $ is the function given by $ \phi_{0} := 1_{B_{0}(0)} $. Then $ \phi_{0} $\, satisfies $ \sum_{\beta \in I_{p}}\phi_{0}(\cdot -\beta) = 1 $. Moreover, we have $ a_{n} = S_{h}^{n}\delta $\, and let $ b_{n} $\, be the sequence given by
$$ b_{n}(\alpha) = \max_{\gamma \in supp(h)} \lvert \nabla_{\gamma} a_{n}(\alpha)\rvert,\quad \alpha \in I_{p}.$$
We claim that there exists a positive constant $ C $\, independent of $ n $, such that
\begin{equation}\label{eq4.8}
 \lVert f_{n+1}-g_{n} \rVert_{q} \leq Cp^{-n/q}\lVert b_{n} \rVert_{q}, \quad 1 \leq q \leq \infty. 
\end{equation}
If $ \rho_{q} < p^{1/q} $, then we can find $ r,\, 0 < r < 1 $, such that $ \rho_{q} < rp^{1/q} $. By the definition of $ \rho_{q} $\, and Theorem \ref{th3.5}, we see that $ \lVert b_{n} \rVert_{q} < rp^{1/q} $\, is valid for sufficiently large $ n $. Hence, there exists a positive constant $ C_{0} $\, such that $ \lVert b_{n} \rVert_{q} \leq C_{0} (rp^{1/q})^{n} $\, for all $ n \geq 1 $. If we choose $ \phi $\, to be $ \phi_{0} $, then it follows from \eqref{eq4.8} that
$$ \lVert f_{n+1}-f_{n} \rVert_{q} \leq CC_{0}r^{n}, \quad n = 1, 2, \ldots  . $$
This shows that the sequence $ f_{n} $\, converges to a function $ f $\, in the $ L_{q} $-norm. Furthermore, \eqref{eq4.8} tells us that $ \lVert f_{n+1}-g_{n} \rVert_{q} \rightarrow 0 $\, as $ n \rightarrow \infty $. However,
$$ \lVert g_{n}-f \rVert_{q} \leq \lVert g_{n}-f_{n+1} \rVert_{q} + \lVert f_{n+1}-f \rVert_{q}, $$
by the triangle inequality. Therefore, $ \lVert g_{n}-f \rVert_{q} \rightarrow 0 $\, as $ n \rightarrow \infty $. Thus, it suffices to
prove \eqref{eq4.8}. \par 
 For a sequence $ \lambda \in l_{\infty}(I_{p}) $\, and a subset $ G $\, of $ \mathbb{Q}_{p} $, we use $ \lVert \lambda \rVert_{\infty}(G) $\, to denote the supremum of $ \lambda $\, on the set $ I_{p} \cap G $. Moreover, for $ n = 1, 2, \ldots $, let
$$ X_{n}(\eta) := A^{-n}(B_{0}(0)+\eta), \quad \eta \in I_{p}. $$
Let $ x $\, be a point in $ \mathbb{Q}_{p} $. By \eqref{eq3.1} we have
$$ f_{n+1}(x) = \sum_{\beta \in I_{p}} a_{n+1}(\beta)\phi_{0}(A^{n+1}x-\beta) \text{   and   } g_{n}(x) = \sum_{\alpha \in I_{p}} a_{n}(\alpha)\phi(A^{n}x-\alpha). $$
Since $ \sum_{\beta \in I_{p}}\phi_{0}(\cdot -\beta) = 1 $\, and $ \sum_{\alpha \in I_{p}}\phi(\cdot -\alpha) = 1 $, it follows that
\begin{equation}\label{eq4.9}
 f_{n+1}(x) -g_{n}(x) = \sum_{\alpha \in I_{p}} \sum_{\beta \in I_{p}} [a_{n+1}(\beta)-a_{n}(\alpha)]\phi(A^{n}x-\alpha)\phi_{0}(A^{n+1}x-\beta).
\end{equation}
In the above sum we only have to consider those terms for which $ \phi(A^{n}x -\alpha) \neq 0 $\, and $ \phi_{0}(A^{n+1}x-\beta) \neq 0 $.\par 
Let $ x \in X_{n}(\eta) $, where $ \eta \in I_{p} $ is fixed for the time being. Suppose $ \phi(A^{n}x-\alpha) \neq 0 $. Then we have $ A^{n}x -\eta \in B_{0}(0) $ and $ A^{n}x -\alpha \in supp(\phi) $. It follows that
\begin{equation}\label{eq4.10}
 \alpha = \eta +(A^{n}x-\eta)-(A^{n}x-\alpha) \in \eta+B_{0}(0)-supp(\phi).
\end{equation}
Suppose $ \phi_{0}(A^{n+1}x-\beta) \neq 0 $. Then $ A^{n+1}x-\beta \in  supp(\phi_{0}) $. This, in connection with $ A^{n}x -\alpha \in supp(\phi) $, yields
\begin{equation}\label{eq4.11}
 A\alpha-\beta=(A^{n+1}x-\beta)-A(A^{n}x-\alpha) \in supp(\phi_{0})-A \,supp(\phi).
\end{equation}
Moreover, Theorem \ref{th4.1} tells us that \eqref{eq4.7} implies $ \sum_{\zeta \in I_{p}} h(\beta -A\zeta) = 1 $ for all $ \beta \in I_{p} $; hence, we have
\begin{equation}\label{eq4.12}
 a_{n+1}(\beta)-a_{n}(\alpha) = \sum_{\zeta \in I_{p}} h(\beta-A\zeta)[a_{n}(\zeta)-a_{n}(\alpha)].
\end{equation}
We observe that $ h(\beta -A\zeta) \neq 0 $ implies $ \beta -A\zeta \in supp(h) $. This, together with \eqref{eq4.11}, gives
\begin{equation}\label{eq4.13}
 A(\alpha-\zeta) = (A\alpha-\beta)+(\beta-A\zeta) \in supp(\phi_{0})-A\,supp(\phi)+supp(h).
\end{equation}
In light of \eqref{eq4.10} and \eqref{eq4.13}, there exists a positive integer $ N $\, such that both $ \alpha $\, and $ \zeta $\, belong to $ \eta + B_{N}(0) $, provided $ \phi(A^{n}x-\alpha) \neq 0, \phi_{0}(A^{n+1}x-\beta)\neq 0 $, and $ h(\beta-A\zeta) \neq 0 $. However, $ a_{n}(\zeta)-a_{n}(\alpha) $ can be written as a sum of finitely many terms of the form $ \nabla_{\gamma}a_{n}(\nu) $, where $ \nu \in \eta + B_{N}(0) \cap I_{p} $ and $ \gamma \in supp(h)\setminus\{0\} $. Therefore, \eqref{eq4.12} tells us that there exists a positive constant $ C $\, independent of $ n $\, such that
\begin{equation}\label{eq4.14}
 \lvert a_{n+1}(\beta) - a_{n}(\alpha) \rvert \leq C \lVert b_{n} \rVert_{\infty}(\eta+B_{N}(0))
\end{equation}
provided $ \phi(A^{n}x -\alpha)\phi_{0}(A^{n+1}x -\beta) \neq 0 $ for some $ x \in X_{n}(\eta) $.
We observe that $ \sum_{\beta \in I_{p}} \lvert \phi_{0}(A^{n+1}x-\beta) \rvert = 1 $. Consequently, by \eqref{eq4.9} and \eqref{eq4.14} we obtain
\begin{equation}\label{eq4.15}
 \lvert f_{n+1}(x)-g_{n}(x) \rvert \leq C \lvert \phi \rvert^{\circ} (A^{n}x) \lVert b_{n} \rVert_{\infty}(\eta+B_{N}(0)) \quad \text{for }x \in X_{n}(\eta),
\end{equation}
where $ \lvert \phi \rvert^{\circ} $ denotes the 1-periodization of $ \lvert \phi \rvert $:
$$ \lvert \phi \rvert^{\circ}(x) := \sum_{\alpha \in I_{p}} \lvert \phi(x-\alpha) \rvert, \quad x \in \mathbb{Q}_{p}. $$
In the case $ q = \infty $, $ \phi $\, is a continuous function with compact support; hence, there exists a constant $ C_{1} > 0 $ such that $\lvert \phi \rvert^{\circ}(x) \leq C_{1} $ for all $ x \in \mathbb{Q}_{p} $. It follows from \eqref{eq4.15} that $ \lVert f_{n+1}-g_{n} \rVert_{\infty} \leq C C_{1}\lVert b_{n} \rVert_{\infty} $. This proves \eqref{eq4.8} for the case $ q = \infty $.\par 
For $ 1 \leq q < \infty $, we deduce from \eqref{eq4.15} that
$$ \int_{X_{n}(\eta)} \lvert f_{n+1}(x)-g_{n}(x) \rvert^{q} dx \leq C^{q}\left( \int_{X_{n}(\eta)} [\lvert \phi \rvert^{\circ} (A^{n}x)]^{q}dx\right) \sum_{\alpha \in \eta+B_{N}(0)} \lvert b_{n}(\alpha) \rvert^{q}. $$
Since $ \phi \in L_{q}(\mathbb{Q}_{p}) $ is compactly supported, we have
$$ \lvert \phi \rvert^{\circ}(x) := \sum_{\alpha \in I_{p}\cap(B_{0}(0)-supp(\phi))} \lvert \phi(x-\alpha) \rvert, \quad x \in B_{0}(0). $$
Hence, $ C_{2} := \int_{B_{0}(0)}\lvert \phi \rvert^{\circ}(x) ^{q} dx < \infty $. Consequently,
$$ \int_{X_{n}(\eta)} [\lvert \phi \rvert^{\circ} (A^{n}x)]^{q}dx = p^{-n} \int_{\eta+B_{0}(0)} [ \lvert \phi \rvert^{\circ} (x) ]^{q} dx = C_{2}p^{-n}. $$ 
Finally, we obtain
$$ \begin{aligned}
    \lVert f_{n+1}-g_{n} \rVert_{q}^{q} & = \sum_{\eta \in I_{p}} \int_{X_{n}(\eta)} \lvert f_{n+1}(x)-g_{n}(x) \rvert^{q} dx \\
    & \leq C^{q}C_{2}p^{-n} \sum_{\eta \in I_{p}} \sum_{\alpha \in \eta+B_{N}(0)} \lvert b_{n}(\alpha) \rvert^{q}.
   \end{aligned}
 $$
However,
$$ \sum_{\eta \in I_{p}} \sum_{\alpha \in \eta+B_{N}(0)} \lvert b_{n}(\alpha) \rvert^{q} = \sum_{\alpha \in I_{p}} \lvert b_{n}(\alpha) \rvert^{q} \sum_{\eta \in \alpha+B_{N}(0)} 1 = p^{N+1} \sum_{\alpha \in I_{p}} \lvert b_{n}(\alpha) \rvert^{q}. $$
The preceding discussion establishes that
$$ \lVert f_{n+1}-g_{n} \rVert_{q} \leq C_{3} p^{-n/q}\lVert b_{n} \rVert_{q} $$
for some constant $ C_{3} > 0 $.
\end{proof}
\par 
Suppose $ K = supp(h) $. Then $ K $\, is an admissible set for every $ A_{\epsilon},\, \epsilon \in  E $, and $ l(K) $\, contains $ \nabla_{\gamma}\delta $\, for $ \gamma \in K \setminus \{0\} $. Let
$$ V := \{ v \in l(K): \sum_{\alpha \in I_{p}}v(\alpha) = 0 \}. $$
If $ \sum_{\beta \in I_{p}}h(\alpha-A\beta) = 1 $, then $ V $\, is invariant under every $ A_{\epsilon},\, \epsilon \in E $. Thus, we can
restate Theorem \ref{th4.2} as follows.
\begin{theorem}\label{th4.3}
 Under the conditions of Theorem \ref{th4.2}, the subdivision scheme
associated with $ h $\, converges to a compactly supported orthogonal function in the $ L_{q} $-norm $ (1 \leq q \leq \infty) $\, if and only if the following two conditions are satisfied:
\begin{enumerate}
 \item $ \sum_{\beta \in I_{p}}h(\alpha-A\beta) = 1, \qquad \forall \alpha  \in I_{p} $.
 \item $ \rho_{q}(\{A_{\epsilon \lvert V(v)} : \epsilon \in E \}) <p^{1/q} $.
\end{enumerate}
\end{theorem}
\begin{proof}
 For $ \gamma \in K \setminus \{0\} $, $ \nabla_{\gamma}\delta \in V $\, span $ V $. Let $ \mathcal{A} := \{A_{\epsilon \lvert V} : \epsilon \in E\} $. Then we have
$$ \lVert \mathcal{A}^{n}\nabla_{\gamma}\delta \rVert_{q} = \lVert \nabla_{\gamma}S_{h}^{n}\delta \rVert_{q}. $$
This shows that
$$ \rho_{q}(\{A_{\epsilon \lvert V(v)} : \epsilon \in E \}) = \max_{\gamma \in K\setminus \{0\}} \{ \lim_{n \rightarrow \infty} \lVert \nabla_{\gamma}S_{h}^{n}\delta \rVert_{q}^{1/n}\}. $$
Thus, Theorem \ref{th4.3} follows from Theorem \ref{th4.2} at once.
\end{proof}

\begin{example}
  Let $ p = 3 $\, and $ h(0) = h(1/3) = h(2/3) = 1, h(\alpha) = 0 $, for all other $ \alpha \in I_{3} $. Here $ E = \{0,1/3,2/3\} $\, and $ K = \{ 0, 1/3,2/3 \} $. Then $ V = \{ v \in l(K) : \sum_{\alpha \in I_{3}}v(\alpha) = 0 \} $\, is the minimal invariant subspace generated by 
  $$ v_{1}(\alpha) = \delta-\delta_{1/3} = \begin{cases}
                                    1 & ; \alpha = 0 \\
                                    -1 & ; \alpha = 1/3 \\
                                    0 & ; \text{otherwise}
                                  \end{cases}
\text{     and      } 
v_{2}(\alpha) = \delta-\delta_{1/3} = \begin{cases}
                                    1 & ; \alpha = 0 \\
                                    -1 & ; \alpha = 2/3 \\
                                    0 & ; \text{otherwise}
                                  \end{cases}. $$
Then $ A_{0} = A_{1/3} = A_{2/3} = 0 $. Thus $ \rho_{q}(\{A_{\epsilon \lvert V}: \epsilon \in E\}) = 0 $. Thus the subdivision scheme associated with $h $\, is convergent.
\end{example}
\begin{example}
 Let $ p = 2 $. Suppose $ h(0) = h(2/4) = (1+i)/2 $, and $ h(1/4)= h(3/4) = (1-i)/2 $, $ h(\alpha) = 0 $, for all other $ \alpha \in I_{2} $. Here $ E = \{0,1/2\} $\, and $ K = \{ 0, 1/4,2/4,3/4 \} $. Then $ V = \{ v \in l(K) : \sum_{\alpha \in I_{2}}v(\alpha) = 0 \} $\, is the minimal invariant subspace generated by 
  $$ v_{1}(\alpha) = \delta-\delta_{1/4} = \begin{cases}
                                    1 & ; \alpha = 0 \\
                                    -1 & ; \alpha = 1/4 \\
                                    0 & ; \text{otherwise}
                                  \end{cases}
\text{     and      } 
v_{2}(\alpha) = \delta-\delta_{2/4} = \begin{cases}
                                    1 & ; \alpha = 0 \\
                                    -1 & ; \alpha = 2/4 \\
                                    0 & ; \text{otherwise}
                                  \end{cases},
                                  $$
     and $$     
v_{3}(\alpha) = \delta-\delta_{3/4} = \begin{cases}
                                    1 & ; \alpha = 0 \\
                                    -1 & ; \alpha = 3/4 \\
                                    0 & ; \text{otherwise}
                                  \end{cases}.
                                  $$
                                  
Then $$ A_{0} = A_{1/2} = \begin{bmatrix}
                 0 & 0 & 0 \\
                 i & 0 & i \\
                 0 & 0 & 0
                \end{bmatrix}.            
  $$ 
  Thus $ \rho_{q}(\{A_{\epsilon \lvert V}: \epsilon \in E\}) = 0 $. Thus the subdivision scheme associated with $h $\, is convergent.
\end{example}
\begin{example}
 Let $ p =2 $\, and $ h(0) = 2, h(\alpha) = 0 $, for all other $ \alpha \in I_{2} $. Then $ \sum_{\beta \in I_{2}} h(-A\beta) = 2 $. 
  That is $ \sum_{\beta \in I_{p}}h(\alpha-A\beta) \neq 1,\, \forall \alpha  \in I_{p} $. Thus by Theorem \ref{th4.3}, the subdivision scheme associated with this $ h $\, is not convergent. 
\end{example}

\section{Smoothness of refinable and wavelet functions}

In the case of the Euclidean space $ \mathbb{R}^{n} $, we use $ C^{m} \equiv C^{m}(\mathbb{R}^{n}) $\, to denote the function space of all $ m $-order continuous differentiable functions. In order to describe the smoothness of the functions deﬁned on $ \mathbb{R}^{n} $, we need $ 1 $-order continuity modulus as well as $ 2 $-order continuity modulus. That is, we need Lipschitz space as well as generalized Lipschitz space \cite{rqj2}. But for $ p $-adic fields, we just need $ 1 $-order continuity modulus, and just Lipschitz space. \par 
For any $ n \in \mathbb{Z} $\, we deﬁne,
$$ E_{n} := \{f : (f \text{ is locally constant})\, f(x - h) = f(x),\, \forall x \in \mathbb{Q}_{p},\, h \in B_{-n}(0) \}. $$
For any $ q \in [1, \infty] $, let $ E_{n,q} := E_{n} \cap L_{q}(\mathbb{Q}_{p}) $.
\begin{definition}\cite{ssp,wys}
 The best approximation of a function $ f \in L_{q}(\mathbb{Q}_{p}) $\, by functions in $ E_{n,q},\, n \in \mathbb{N} $\, is the number
\begin{equation}\label{eq5.1}
 E_{n}^{q}(f) := \inf \{ \lVert f - \Phi \rVert_{q}: \Phi \in E_{n,q} \}. 
\end{equation}
\end{definition}
 \begin{definition}\cite{ssp,wys}
  For $ f \in L_{q}(\mathbb{Q}_{p}) $\, and $ n \in \mathbb{N} $, let 
  \begin{equation}\label{eq5.2}
   \omega_{q}(f,n) := \sup \{ \lVert f-f(\cdot-h) \rVert_{q}: h \in B_{-n}(0) \}.
  \end{equation}
The sequence of numbers $ \{ \omega_{q}(f,n) \}_{n\in \mathbb{N}} $\, is called the modulus of continuity of $ f $\, in the space $ L_{q}(\mathbb{Q}_{p}) $. In general we have,
$$  \omega_{q}(f,\delta) := \sup_{h \in \mathbb{Q}_{p}, \lvert h \rvert_{p} \leq \delta} \{ \lVert f-f(\cdot-h) \rVert_{q} \}, \quad \delta >0. $$
 \end{definition}
\begin{definition}\cite{ssp,wys}
 A function $ f(x) $\, belongs to the Lipschitz space $ Lip(\alpha, L_{q}(\mathbb{Q}_{p})) $, if $ f \in L_{q}(\mathbb{Q}_{p}) $\, and for some constant $ c = c(f) > 0 $\, we have
\begin{equation}\label{eq5.3}
 \omega_{q}(f,n) \leq cp^{-\alpha n}, \qquad n \in \mathbb{N}.
\end{equation}
\end{definition}
The following theorem states that these three definitions are equivalent under certain conditions.
\begin{theorem}\label{th5.1}\cite{wys}
 For function spaces $ L_{q}(\mathbb{Q}_{p}) $, if $ s \geq 0 $, then the following statements are equivalent 
 \begin{enumerate}
  \item $ D^{s}f \in Lip(\alpha, L_{q}(\mathbb{Q}_{p})), \, \alpha > 0 $.
  \item $ \omega_{q}(D^{s}f,n) = O(p^{-\alpha n}),\, n \rightarrow \infty $.
  \item $ E_{n}^{q}(f) = O(p^{-n(\alpha+s)}),\, n \rightarrow +\infty $.
 \end{enumerate}
\end{theorem}
\begin{definition}\cite{wys,sc}
 The H\"{o}lder type space on $ p $-adic fields, denoted by $ C^{\sigma}(\mathbb{Q}_{p}) $\, is the collection of all functions such that;
\begin{equation}\label{eq5.4}
 f \in C^{\sigma}(\mathbb{Q}_{p}) \Leftrightarrow D^{\lambda}f \in C^{\sigma-\lambda}(\mathbb{Q}_{p}), \quad \forall \, 0 \leq \lambda \leq \sigma.
\end{equation}
where $ D^{\lambda}f $\, is the pseudo differential operator defined by \eqref{eq1.2}. 
\end{definition}
The next theorem gives the relationship between the H\"{o}lder type spaces and the Lipschitz space on $ \mathbb{Q}_{p} $.
\begin{theorem}\label{th5.2}\cite{sc}
 For a $ p $-adic field $ \mathbb{Q}_{p} $, we have $ Lip(\alpha, L_{q}(\mathbb{Q}_{p})) = C^{\alpha}(\mathbb{Q}_{p}), \, \alpha > 0 $.
\end{theorem}
\begin{definition}
 The optimal smoothness of a vector $ f \in L_{q}(\mathbb{Q}_{p}) $\, in the $ L_{q} $-norm is described by its critical exponent $ \nu_{q}(f) $\, defined by
 \begin{equation}\label{eq5.5}
  \nu_{q}(f) := \sup \{ \nu : f \in Lip(\nu,L_{q}(\mathbb{Q}_{p})) \}.
 \end{equation}
\end{definition}
\begin{theorem}\label{th5.3}
 Every locally constant, compactly supported function in $ L_{q}(\mathbb{Q}_{p}) $\, is infinitely smooth. 
\end{theorem}
\begin{proof}
 Let $ \phi $\, be a locally constant compactly supported function in $ L_{q}(\mathbb{Q}_{p}) $. Then $ supp(\phi) = B_{N}(0) $\, for some $ N \in \mathbb{Z} $, and $ \phi(x-h)=\phi(x), \, \forall h \in B_{M}(0) $\, for some $ M \in \mathbb{Z} $.
We have \begin{equation}\label{eq5.6}
 \cdots \supset B_{2}(0) \supset B_{1}(0) \supset B_{0}(0) \supset B_{-1}(0) \supset B_{-2}(0) \supset \cdots. 
\end{equation}
Thus, as $ n $\, is increases, $ B_{-n}(0) \supset B_{N}(0) $\, and $ B_{-n}(0) \supset B_{M}(0) $. Thus $ \phi(x) = \phi(x-h),\, h \in B_{-n}(0) $. 
 $$ \lVert \phi -\phi(\cdot-h) \rVert_{q} = \left(\int_{B_{N}(0)} \lvert \phi(x)-\phi(x-h) \rvert^{q} dx\right)^{1/q} = 0 , \text{  as } n \rightarrow \infty.                                                      $$
   That is condition (2) of Theorem \ref{th5.1} is satisfied for all $ \alpha >0 $. Thus from Theorem \ref{th5.1}, we can see that $ \phi $\, is infinitely smooth. 
\end{proof}

\section{Conclusion}
The definitions of subdivision operators and subdivision scheme associated with a finitely supported refinement mask on $ I_{p} $\, are given in this work. The $ L_{q} $-convergence of the subdivision scheme is characterized in terms of the $ q $-norm joint spectral radii of a collection of operators associated with the refinement mask. We also provided some examples to illustrate the general theory. Finally, we provided the detailed analysis of smoothness of functions in $ L_{q}(\mathbb{Q}_{p}) $. In particular, we proved that any compactly supported function in $ L_{q}(\mathbb{Q}_{p}) $\, is infinitely smooth.

\section*{Acknowledgement}
We are very grateful to the authors of the articles in the references.

\bibliographystyle{IEEEtran}
\bibliography{article}

\end{document}